\documentclass[a4,10pt,reqno,fleqn]{amsart}

\title[Projective tensor product of $C^*$-algebras] {On Banach space
  projective tensor product of $C^*$-algebras}

\author[V. P. Gupta]{Ved Prakash Gupta}
\address{School of Physical
  Sciences,\\ Jawaharlal Nehru University\\ New Delhi-110067, INDIA.}
\email{vedgupta@mail.jnu.ac.in}

\author[R. Jain]{Ranjana Jain}
\address{Department of
  Mathematics\\ University of Delhi\\ Delhi-110007, INDIA.}
\email{rjain@maths.du.ac.in}

\thanks{The first named author was supported by a UPE-II
  project (with Id 228) of Jawaharlal Nehru University, New Delhi}

\usepackage[all]{xy}
\usepackage{amsmath,amsfonts,amssymb,amsthm,a4,hyperref}

\usepackage[margin=3.5cm]{geometry}
\usepackage{hyperref}
\usepackage{cleveref}
\usepackage{setspace}
\usepackage{enumitem}
\usepackage{verbatim}

\newtheorem{theorem}{\sc Theorem}[section]
\newtheorem{cor}[theorem]{\sc Corollary}
\newtheorem{prop}[theorem]{\sc Proposition}
\newtheorem{lemma}[theorem]{\sc Lemma}
\numberwithin{equation}{section}

\theoremstyle{remark}

\newcommand{\oop}{\widehat\otimes}
\newcommand{\seq}{\subseteq}
\newcommand{\oh}{\otimes^h}
\newcommand{\omin}{\otimes^{\min}}
\newcommand{\obp}{\otimes^\gamma}

\newcommand{\ra}{\rightarrow}
\newcommand{\ot}{\otimes}
\newcommand{\ota}{\otimes^{\alpha}}

\newcommand{\Z}{\mathcal{Z}}

\newcommand{\m}{\mathcal{M}}

\newcommand{\C}{\mathbb{C}}
\newcommand{\N}{\mathbb{N}}

\newcommand{\ol}{\overline}
\newcommand{\mcal}{\mathcal}

\begin{document}


\begin{abstract}
  We analyze certain algebraic structures of the Banach space
  projective tensor product of $C^*$-algebras which are comparable
  with their known counterparts for the Haagerup tensor product and
  the operator space projective tensor product of
  $C^*$-algebras. Highlights of this analysis include (a) injectivity
  of the Banach space projective tensor product when restricted to the
  tensor products of $C^*$-algebras, (b) detailed structure of closed
  ideals of $A \obp B$ in terms of those of $A$ and $B$, (c)
  identification of certain spaces of ideals of $A \obp B$ in terms of
  those of $A$ and $B$ from the perspective of hull-kernel topology,
  and (d) identification of the center of $A \obp B$ with $\Z(A) \obp
  \Z(B)$, where  $A$ and $B$ are  $C^*$-algebras.
 \end{abstract}

\maketitle

\section{ Introduction}

Around 1960s, Gelbaum initiated the study of certain spaces of ideals
of the Banach space projective tensor product $A \obp B$ of Banach
algebras $A$ and $B$, who was then followed by Laursen and Tomiyama -
see \cite{gelbaum1, gelbaum2, gelbaum3, laur, tom} and the references
therein. They focussed mainly on analyzing the spaces of maximal
modular ideals and primitive ideals of $A \obp B$ in terms of those of
$A$ and $B$, and to analyze which properties of Banach algebras are
passed on to their tensor products, where $A$ or $B$ or both is/are
commutative.  However, not much was discovered about the structure of
general ideals of $A \obp B$ in terms of those of $A$ and $B$. A
similar hardship has been observed in understanding the ideal
structure of $A \omin B$ as well, for $C^*$-algebras $A$ and $B$.

On the other hand, the analysis of various algebraic structures of the
Haagerup tensor product $A \oh B$ and the operator space projective
tensor product $A \oop B$, of $C^*$-algebras $A$ and $B$, has been
carried out extensively during last three decades (see \cite{ass,
  arc3, JK, jk11, JK-edin, jk-13, kr-1, kr-2, KS}).  An important and
useful development of this project has been the discovery of the
connection that exists between the structures of centers and ideals of
$A \oh B$ and $A \oop B$ in terms of those of $A$ and $B$, all of
which was pioneered by the remarkable work \cite{ass} of Allen,
Sinclair and Smith, wherein they study closed ideals of $A \oh B$. This
article aims at obtaining similar results for $A \obp B$, for
$C^*$-algebras $A$ and $B$.

A closer look reveals that a crucial ingredient that helped to
establish the relationships mentioned above was the injectiviy of
$\oh$ and a partial injectivity of $\oop$ (see \Cref{ideals}, below),
along with an exactness type property (as in \Cref{kernel-phi-ot-psi})
exhibited by both tensor norms. It is known that $\obp $ is not
injective; and, might be due to the lack of any partial injectivity
result, not much could be known about the algebraic structure of $A
\obp B$.  However, exploiting a work of Diestel et al. \cite{puglisi}, the
second named author had demonstrated in her Ph.D. thesis \cite{jain}
that, when restricted to the tensor products of $C^*$-algebras, $\obp$
observes a better partial injectivity than $\oop$, as is re-produced
in Section 2 below. This turns out to have beautiful consequences in
the study of algebraic structures of $A \obp B$, which is achieved
primarily by employing a similar line of treatment as in
\cite{ass}. We present here few of those consequences and are quite
hopeful that we can deduce more.

For any algebra $A$, let $\m(A)$ (resp., $\m_m(A)$) denote the set of
maximal ideals (resp., maximal modular ideals) of $A$. In his first
article to analyze spaces of ideals of $A \obp B$, Gelbaum
(\cite{gelbaum1}) showed that if $A$ and $B$ are commutative Banach
algebras, then there is a homeomorphism between $\m_m(A) \times
\m_m(B)$ and $\m_m(A \obp B)$, when the spaces are equipped with
$w^*$-topologies. Then, in \cite{gelbaum3}, he replaced the condition
of commutativity by existence of unity and obtained an injective map
from $\m(A) \times \m(B)$ into $\m(A \obp B)$ which turned out to be
closed and continuous with respect to the hull-kernel
topology. Surjectivity is still an open question in this case. Almost
simultaneously, Laursen (\cite{laur}) dropped the commutativity of one
of the Banach algebras and established that the spaces $\m_m(A) \times
\m_m(B)$ and $\m_m(A \obp B)$ are homeomorphic with respect to
hull-kernel topology.  For arbitrary $C^*$-algebras $A$ and $B$ (not
necessarily unital or commutative), based on the ideal structue of $A
\obp B$ discussed above, we establish that there is a homeomorphism
from $Id'(A) \times Id'(B)$ onto its image (which is also dense) in
$Id'(A \obp B)$, with respect to $\tau_w$-topology, where $Id'(A)$
denotes the set of proper closed ideals of $A$.  Moreover, this map
restricts to a homeomorphism from $\m_m(A) \times \m_m(B)$ (resp., $\m(A) \times \m(B)$) onto
$\m_m(A \obp B)$ (resp., $\m(A \obp B)$ )with respect to the hull-kernel topology.

Here is a brief overview of the topics discussed in this article. As
mentioned above, Section 2 is devoted towards establishing a partial
injectivity of $\obp$ when restricted to tensor products of
$C^*$-algebras, which is used throughout the article.  Section 3 is
the soul of this article which provides a thorough discussion of ideal
structure of $A \obp B$ in terms of ideals of $C^*$-algebras $A$ and
$B$. Among other things, we show that every closed ideal of $A \obp B$
contains a product ideal; that if $A$ or $B$ has finitely many closed
ideals, then every closed ideal of $A \obp B$ is a sum of product
ideals, and we identify minimal (resp., maximal and maximal modular)
ideals of $A \obp B$ in terms of those of $A$ and $B$. In Section 4,
{we identify spaces of proper (resp., maximal and maximal modular)
  ideals of $A \obp B$ in terms } of those of $A$ and $B$ from the
perspective of hull-kernel topology. Finally, Section 5 provides
another application of partial injectivity of $\obp$, where we provide
an identification of the center of $A \obp B$ with $\Z(A) \obp \Z(B)$.

\section{Injectivity of the Banach space projective norm on tensor products
  of $C^*$-algebras}

Recall that, a norm $\|\cdot \|_{\alpha}$ on the
 algebraic tensor product $A \ot B$ of a pair of $C^*$-algebras $A$
 and $B$ is said to be
\begin{enumerate}
\item a {\em cross norm} if $\|a \ot b\|_\alpha = \|a\|\,
  \|b\|$ for all $a \in A$, $b \in B$,
\item an {\em algebra norm} if $\|w\, z \|_{\alpha}  \leq \| w \|_{\alpha}
  \, \|z\|_{\alpha} $ for all $w, z \in A \ot B$, and 
\item a {\em tensor norm} if $\| \cdot \|_{\lambda} \leq \| \cdot
  \|_{\alpha} \leq \| \cdot \|_{\gamma}$, where $\lambda$ and $\gamma$
  are the Banach space injective and projective norms, respectively.
\end{enumerate}
Clearly, $A \otimes^{\alpha} B$, the completion of $A \ot B$ with
respect to any algebra norm $\|\cdot\|_\alpha$, is a Banach
algebra. From an algebraic point of view, among the well analyzed tensor
products in literature are the $C^*$-minimal tensor product ($\omin$),
the  Haagerup tensor product ($\oh$), the operator
space projective tensor product ($\oop$) { and the Banach space
  projective tensor product ($\ot^\gamma$)}.  We refer the reader to
\cite{ERbook, Tak} and \cite{ryan} for their definitions and
essential properties. All these norms are cross algebra tensor
norms and yield Banach algebras. In fact, for $C^*$-algebras $A$ and
$B$, the canonical involution on $A \ot B$ is not
isometric \footnote{We follow the convention that the involution on a
  Banach $*$-algebra is an isometry.} with respect to $\oh$ (see
\cite{ass}); however, by the very definition of
$\|\cdot\|_\gamma$, $A \ot^\gamma B$ is a Banach $*$-algebra; and, by
\cite{Kum01}, so is $A \oop B$. 

The $C^*$-minimal tensor product is known to be
injective and so is the Haagerup tensor product (see \cite{ass, ERbook}). On
the other hand, neither the Banach space projective tensor product nor the operator space projective
tensor product is injective. But in few cases they behave well. For
instance, Kumar (\cite{Kum01}) proved that the tensor product of
ideals of $C^*$-algebras can still be embedded nicely.

\begin{prop}(\cite[Theorem  5]{Kum01})\label{ideals}
Let $A$ and $B$ be $C^*$-algebras, and $I$ and $J$ be closed ideals
in $A$ and $B$, respectively. Then the identity map on $I \ot J$
extends to an isometric algebra map from the Banach algebra $I
\oop J$ onto the product ideal $\ol{I \ot J} \subseteq A
\oop B$.
\end{prop}

The situation with the Banach space projective tensor product of
$C^*$-algebras is even better (\Cref{obp-injective} below) as was proved
in the doctoral thesis \cite{jain} of the second named author. Given
its usefulness, we include the details for convenience of present and
future references.

We first recall some concepts whose details can be found in
\cite{puglisi} and the references therein. The {\it strong*-topology}
on a Banach space $X$ is defined to be the locally convex topology
generated by the seminorms $x \mapsto \|Tx\|$ for bounded linear maps
$T$ from $X$ into Hilbert spaces. Also, an operator $T: A\ra X$, where
$A$ is a $C^*$-algebra and $X$ is a Banach space, is said to be {\it
  $p\text{-}C^*$-summing} for $p>0$ if there exists a constant $k$
such that for every finite sequence $(a_1,a_2,\dots, a_n)$ of
self-adjoint elements in $A$,
$$\left( \sum_{i=1}^n \|T(a_i)\|^p \right)^{1/p} \leq k \Bigg\| \left(
\sum_{i=1}^n |a_i|^p \right)^{1/p} \Bigg\|, $$ where for $a\in A$,
$|a| := (\frac{aa^*+ a^*a}{2})^{1/2}$. We first collect few useful results.

\begin{prop}  \cite[$\S 3$]{puglisi}:\label{2-summ-strong}
  For a $C^*$-algebra $A$ and Banach space $X$, an operator $T:A
\rightarrow X$ is 2-$C^*$-summing if and only if it is strong*-norm
continuous.
\end{prop}

\begin{theorem}\cite[Proposition 2.1]{haag}\label{b3}
 Every bounded linear map $T:A\ra B^*$, $A$ and $B$ being
 $C^*$-algebras, can be factored through a Hilbert space, that is,
 there exists a Hilbert space $H$, and bounded linear maps $u:A \ra H
 $ and $v: H \ra B^*$ such that $T = v\circ u$.
\end{theorem}

\begin{theorem}\cite[Theorem 3.2]{puglisi}\label{b2}
 Let $T:X\ra Y$ be a bounded linear operator between Banach spaces $X$
and $Y$. Then $T$ is strong*-norm continuous if and only if $T$
factors through a Hilbert space.
\end{theorem}

\begin{theorem}\cite[Theorem 3.6]{puglisi}\label{b1}
Let $A$ be a $C^*$-algebra, $B$ be a $C^*$-subalgebra of $A$ and
$Y$ be any Banach space. Then every $2\text{-}C^*$-summing operator
$T:B\rightarrow Y$ extends to a norm preserving $2\text{-}C^*$-summing
operator $\widetilde{T}: A \ra Y$.
\end{theorem}

Now, with these results in hand, we are ready to prove our result
regarding the injectivity of Banach space projective norm in the $C^*$ set-up.

\begin{theorem} \label{obp-injective}
Let $A_1$ and $B_1$ be  $C^*$-subalgebras of $C^*$-algebras $A$
and $B$, respectively. Then the identity map on $A_1\ot B_1$ extends
to an isometric $*$-algebra map from $A_1\obp B_1$ onto the closed
$*$-subalgebra $ \ol{A_1 \ot B_1} \subseteq A\obp B$.
\end{theorem}

\begin{proof}
By \cite[Corollary 2.12]{ryan}, $A_1 \obp B_1$ is a subspace of $A
\obp B_1$ if and only if every bounded operator from $A_1$ into
$B_1^*$ extends to an operator of the same norm from $A$ into
$B_1^*$. Let $T$ be one such operator from $A_1$ into $B_1^*$. Then,
by Theorem \ref{b3}, $T$ can be factored through a Hilbert space,
which, by Theorem \ref{b2}, is strong*-norm continuous. By
\Cref{2-summ-strong}, $T$ is $2\text{-}C^*$-summing, therefore, by
Theorem \ref{b1}, $T$ extends to a norm preserving
$2\text{-}C^*$-summing operator $\tilde{T}: A \ra B_1^*$. Thus $A_1
\obp B_1$ is a closed subspace of $A \obp B_1$.  Since $\obp$ is symmetric,
it also follows that $A \obp B_1$ is a closed subspace of
$A \obp B$.
\end{proof}

\section{Ideal structure of $A \obp B$}
 If $\alpha$ is either the Haagerup tensor product or the operator
 space projective tensor product, $A$ and $B$ are $C^*$-algebras with
 $A$ topologically simple, then by \cite[Proposition 5.2]{ass}, and by
 \cite[Theorem 3.8]{JK-edin}, it is known that every closed ideal of
 the Banach algebra $A \ot^\alpha B$ is a product ideal of the form $
 A \ot^\alpha J$ for some closed ideal $J$ in $B$. For $\omin$, we
 obtained, in \cite{GJ}, the following analogue of above results:

\begin{theorem}\cite{GJ}\label{simple-ideal}
 Let $A$ and $B$ be $C^*$-algebras where $A$ is topologically
 simple. If either $A$ is exact or $B$ is nuclear, then every closed
 ideal of the $C^*$-algebra $A \omin B$ is a product ideal of the form
 $ A \omin J$ for some closed ideal $J$ in $B$.
 \end{theorem}

It turns out that the ideal structure of the Banach space projective
tensor product $A \obp B$ also follows on similar lines as in
\Cref{simple-ideal}, provided its ingredients can be established - see
\Cref{obp-ideal} below.  The steps involved are very much on the lines
of \cite{ass} and \cite{JK-edin}, and we avoid mentioning them at
every instance.

Since $\| \cdot \|_{\min}$ and $\|\cdot\|_{{h}}$ are cross-norms and
$\| \cdot \|_{\lambda} \leq \|\cdot \|_{\min} \leq \| \cdot \|_{{h}}
\leq \| \cdot \|_{\gamma}$, where $\| \cdot \|_{\lambda}$ is the
Banach space injective tensor norm, by the Remark on Page $97$ of
\cite{haag}, we have the following crucial embeddings.

\begin{prop}\label{obp-min}\cite{haag}
 Let $A$ and $B$ be $C^*$-algebras. Then the identity map on $A \ot B$
 extends to a contractive injective $*$-homomorphism $i_{\min} : A \obp B \ra A
 \omin B$ and injective homomorphisms $i_{h}: A \obp B \ra A \oh B$
 and $j_{\min}: A \oh B \ra A \omin B$. Also, the  diagram
 \begin{equation}\label{min-h-min}
   \xymatrix{
     A \obp B \ar[dr]^{i_{\min}} \ar[rr]^{i_h}& & A \oh B\ar[dl]^{j_{\min}} \\
     & A \omin B &
     }
   \end{equation}
commutes.
\end{prop}

For a closed ideal $I$ in $A \obp B$, let $I_{\min} := \ol{i_{\min}(I)}
\subseteq A \omin B$. On the lines of \cite[Lemma 4.2]{ass}, Kumar and
Rajpal   \cite{kr-1} proved the following useful result.
\begin{prop}\label{1-tensor-1}\cite{kr-1}
Let $M$ and $N$ be von Neumann algebras and let $I$ be a closed ideal
in $M \obp N$. If $1 \ot 1 \in I_{\min} \subseteq M \omin
N$, then $1 \ot 1 \in I$, and, in particular, $I$ equals $M \obp N$.
  \end{prop}

Analogous to \cite[Theorem 4.4]{ass}, we now prove a theorem that,
along with \Cref{obp-injective}, turns out to be the main
ingredient in the study of ideal structure of Banach space projective
tensor product of $C^*$-algebras.

\begin{theorem}\label{a-ot-b}
  Let $A$ and $B$ be $C^*$-algebras and let $I$ be a closed ideal in
  $A \obp B$. If an elementary tensor $a \ot b \in I_{\min}$, then $a
  \ot b \in I$.
  \end{theorem}

\begin{proof}
We first prove for $a, b \geq 0$.  Suppose $a \ot b \in I_{\min}$ and
is not in $I$. By Hahn-Banach Theorem, there exists a $\varphi \in (A
\obp B)^*$ such that $\varphi (I) = (0)$ and $\varphi (a \ot b ) \neq
0$. It is well known that $(A \obp B)^*$ can be identified canonically
with $B(A, B^*)$, the space of bounded linear maps from $A$ into $B^*$
- see \cite[$\S 2$, page 24]{ryan}. In particular, there exists a
$\Phi \in B(A, B^*)$ such that $\varphi (x \ot y) = \Phi(x)(y)$ for
all $x \in A$ and $ y \in B$. Note that $\Phi^{**}: A^{**} \ra
B^{***}$ is $w^*$-$w^*$continuous and satisfies $\| \Phi^{**} \| =
\|\Phi\| = \| \varphi \|$. Further, the association $A^{**} \ot B^{**}
\ni u \ot v \mapsto \Phi^{**}(u)(v) \in \C$ extends linearly to a
continuous functional on $A^{**} \obp B^{**}$, say,
$\tilde{\varphi}$. Since $A \obp B \subseteq A^{**} \obp B^{**}$
(\cite[Corollary 2.14]{ryan} or \Cref{obp-injective}) ,
$\tilde{\varphi}$ extends $\varphi$ and also $\| \tilde{\varphi}\| =
\| \varphi\|$.

Now, consider the enveloping von Neumann algebras $M := A^{**}, N :=
B^{**}$, and let $ \tilde{I}$ be the closed ideal in $M \obp N$
generated by $I$. We claim that $\tilde{\varphi}(\tilde{I}) = (0)$ as
well. For this, it is enough to show that
\begin{equation}\label{verify}
  \tilde{\varphi} ( (u \ot s) z (v \ot t)) = 0   
\end{equation}
$ \mathrm{for\ all}\ z \in I, u, v \in M\ \mathrm{ and} \ s, t \in N$.
To begin with, let $z \in I \subseteq A \obp B$, $u \in M, v \in A$
and $s, t \in B$. By \cite[Proposition 2.8]{ryan}, there exist bounded
sequences $\{a_i\} \subset A$ and $\{ b_i\} \subset B$ such that $z =
\sum_{i=1}^\infty a_i \ot b_i$. For each $n \geq 1$, consider
$\omega_n \in M^*$ given by $$\omega_n (x) = \sum_{i = 1}^n \Phi^{**}
(x a_i v) ( s b_i t),\ x \in M.$$ Since $\Phi^{**}$ is $w^*$-$w^*$
continuous, $\omega_n$ is $w^*$-continuous, i.e., $\omega_n \in M_*$,
the predual of $M$, for all $n \geq 1$. Also, for $m < n$, we observe that
  \begin{eqnarray*}
  \|\omega_n (x) - \omega_m(x)\| & = & \left\| \tilde{\varphi} \left((x \ot s)
  \left(\sum_{i= m+1}^n a_i \ot b_i\right)(v \ot t)\right)\right\| \\
  & \leq & \|\varphi\| \|x\| \|s\| \left\| \sum_{i=
    m+1}^n a_i \ot b_i\right\|_\gamma  \|v\| \|t\|
\end{eqnarray*}
for all $x \in M$. In particular, $\{\omega_n\}$ is a Cauchy sequence
in $M^*$ with a limit, say, $\omega \in M^*$. Given the actions of
$\omega_n$'s on $M$, the obvious candidate for $\omega$ is given by
$\omega (x) = \sum_{i = 1}^\infty \Phi^{**} (x a_i v) ( s b_i t)$ for
all $ x \in M$. In particular, $\omega \in M_*$. Since, $\varphi(I) =
0$ and $z \in I$, we easily see that $ \sum_{i= 1}^\infty{\varphi} (x
a_iv \ot sb_i t) = 0$ for all $ x \in A$. This implies that $A
\subseteq \ker(\omega)$ and, since $\omega$ is $w^*$- (equivalently,
$\sigma$-weakly) continuous, by von Neumann's Bicommutant Theorem,
we obtain $M = A^{**} = \ol{A}^{w^*} = \ker (\omega)$, i.e.,
\Cref{verify} holds for all $u \in M, v \in A$ and $s, t \in B$.
Repeating the argument by letting $v, s$ and $t$ vary successively, we
conclude that \Cref{verify} holds $ \mathrm{for\ all}\ z \in I, u, v
\in M\ \mathrm{ and} \ s, t \in N$.

With this observation at our disposal, we now show that $a \ot b$ can
be approximated appropriately by elements of $\tilde{I}$ and deduce
that it is annihilated by $\tilde{\varphi}$ to obtain a contradiction.

For each $\epsilon, \nu > 0$, let $p_\epsilon, q_\nu$ be the spectral
projections in $M$ and $N$ associated to $a$ and $b$ for the closed
intervals $[\epsilon, \infty)$ and $[\nu, \infty)$, respectively. Then
    $p_\epsilon M p_\epsilon $ and $q_\nu N q_\nu$ are von Neumann
    subalgebras of $M$ and $N$ with units $p_\epsilon$ and $q_\nu$,
    respectively. In view of the embedding given in \Cref{obp-injective},
    consider the closed ideal $\tilde{I}_{\epsilon, \nu}:= \tilde{I}
    \cap (p_\epsilon M p_\epsilon \obp q_\nu N q_\nu)$. We claim that
    ${(\tilde{I}_{\epsilon, \nu})}_{\min} \subseteq p_\epsilon M
    p_\epsilon \omin q_\nu N q_\nu$ contains the unit $p_\epsilon \ot
    q_\nu$. This will then, by \Cref{1-tensor-1}, yield
    $\tilde{I}_{\epsilon, \nu} = p_\epsilon M p_\epsilon \obp q_\nu N
    q_\nu$ implying that $p_\epsilon a \ot q_\nu b \in
    \tilde{I}_{\epsilon, \nu}$. In particular, $p_\epsilon a \ot q_\nu
    b \in \tilde{I}$ for all $\epsilon, \nu > 0$. And since $
    p_\epsilon a \stackrel{w^*}{\ra} a$ in $M$ and $q_\nu b
    \stackrel{w^*}{\ra} b$ in $N$, we will obtain $\tilde{\varphi} (a
    \ot b) = \lim_{\epsilon \ra 0} \lim_{\nu \ra 0} \tilde{\varphi}
    (p_\epsilon a \ot q_\nu b) = 0$, giving the desired contradiction.

    Towards the claim, note that, by bounded functional calculus,
    $p_\epsilon a $ and $q_\nu b$ are invertible in $p_\epsilon M
    p_\epsilon $ and $q_\nu N q_\nu$, respectively. If $\{ z_n\} $ is
    a sequence in $I$ such that $\{i(z_n)\}$ converges to $a \ot b$ in
    $I_{\min} \subseteq A \omin B \subset M \omin N$, then, again by
    \Cref{obp-injective}, the sequence $\{(p_\epsilon \ot q_\nu) z_n
    (p_\epsilon \ot q_\nu)\}$ is contained in ${\tilde{I}_{\epsilon,
        \nu}}$ and $ j((p_\epsilon \ot q_\nu) z_n (p_\epsilon \ot
    q_\nu)) \ra p_\epsilon a \ot q_\nu b$ in $ p_\epsilon M p_\epsilon
    \omin q_\nu N q_\nu \subseteq M \omin N$, where $j$ is the
    injective homomorphism from $ p_\epsilon M p_\epsilon \obp q_\nu N
    q_\nu$ into $ p_\epsilon M p_\epsilon \omin q_\nu N q_\nu$
    guaranteed by \Cref{obp-min}. This shows that the invertible
    element $ p_\epsilon a \ot q_\nu b$ belongs to
    ${(\tilde{I}_{\epsilon, \nu})}_{\min}$ and hence the unit $
    p_\epsilon \ot q_\nu$ also belongs to $ {(\tilde{I}_{\epsilon, \nu})}_{\min}$.

    Finally, for arbitrary $a$ and $b$, if $ a \ot b \in I_{\min}$
    then, using above positive case, on the lines of last part of
    proof of \cite[Theorem 4.4]{ass}, it can be shown that $ a\ot b
    \in I$.
\end{proof}
\begin{cor}\label{a-ot-b-h}
  Let $A$ and $B$ be $C^*$-algebras and let $I$ be a closed ideal in
  $A \obp B$. If an elementary tensor $a \ot b \in
  I_{h}:=\overline{i_h(I)}^h \subset A \oh B$, then $a \ot b \in I$.
  \end{cor}
\begin{proof}
By \Cref{min-h-min}, we have $\overline{j_{\min}(I_h)}^{\min} =
I_{\min}$ in $A \omin B$. Since $a \ot b \in I_h$, $a \ot b = j_{\min} (a \ot b) \in
I_{\min}$ as well and, therefore, by \Cref{a-ot-b}, $a \ot b \in I$.
  \end{proof}

\Cref{a-ot-b} also allows us to deduce the following, which will be crucial
in the proof of \Cref{obp-ideal}.

\begin{cor}\label{elementary-tensor}
Let $A$ and $B$ be $C^*$-algebras and $I$ be a non-zero closed ideal
in $A \obp B$. Then $I$ contains a non-zero elementary tensor and a product ideal.
\end{cor}
\begin{proof}
  From the Diagram \ref{obp-min}, we see that $I_{\min}$ is a non-zero closed ideal in $A \omin
  B$. So, by \cite[Proposition 4.5]{ass}, $I_{\min}$ contains a
  non-zero elementary tensor, say, $a \ot b$, and then, by
  \Cref{a-ot-b}, $a \ot b \in I$. Also, if $J$ and $K$ are the closed
  ideals of $A$ and $B$, generated by $a$ and $b$, respectively, then
  the product ideal $J \obp K$ is contained in $I$.
  \end{proof}
 This immediately yields the following analogue of
\cite[Corollary 4.7]{ass}.
\begin{cor}\label{obp-faithful}
Let $A$ and $B$ be $C^*$-algebras and $D$ be a Banach algebra. If $\pi
: A \obp B \ra D$ is a bounded homomorphism whose restriction to $A
\ot B$ is faithful, then so is $\pi$.
  \end{cor}

\begin{prop}\label{sum-obp}
Let $A$ and $B$ be $C^*$-algebras. Then a finite sum of closed
product ideals in $A \obp B$ is closed.
   \end{prop}
 \begin{proof}
It is enough to consider the sum of two product ideals.  Let
$J_i, K_i$ be closed ideals in $A$ and $B$, respectively, for $i =1,
2$. By \cite[Proposition 2.4]{dixon}, it is enough to show that $J_1
\obp K_1$ has a bounded  approximate identity. Since every closed
ideal in a $C^*$-algebra possesses a bounded approximate identity, by
\cite[Lemma 3.1]{JK-edin}, $J_1 \obp K_1$ possesses a bounded
approximate identity.
 \end{proof}

From \Cref{obp-injective}, \Cref{sum-obp} and the fact that a finite
sum of closed ideals in a $C^*$-algebra is closed, we easily deduce
the following.

 \begin{cor}\label{finite-sum}
   Let $\{J_i\}_{i = 1}^n$ and $\{ K_j\}_{j=1}^m$ be closed ideals in
   $C^*$-algebras $A$ and $B$, respectively. Then,
   \begin{enumerate}
   \item $(\sum_i J_i) \obp B = \sum_i J_i \obp B$, and
     \item $A \obp (\sum_j K_j) = \sum_j A \obp K_j$.
     \end{enumerate}
   \end{cor}

 Recall that a map $\pi : X \ra Y$ between two Banach spaces is said
 to be a {\em quotient map} if it maps the open unit ball of $X$ onto
 that of $Y$. In particular, a quotient map is surjective. For two
 quotient maps $\varphi_i: X_i \ra Y_i$, $\varphi_1 \ot \varphi_2$
 extends to a quotient map $\varphi_1 \obp \varphi_2 : X_1 \obp X_2
 \ra Y_1 \obp Y_2$ - see \cite[Proposition 2.5]{ryan}.  Anologous to
 \cite[Theorem 2.4]{ass}, \cite[Proposition 7.1.7]{ERbook} and
 \cite[Proposition 3.3]{kr-2}, we obtain the following essential
 result:

 \begin{prop}\label{kernel-phi-ot-psi}
Let $X_i$ and $Y_i$ be Banach spaces and $\varphi_i : X_i \ra Y_i$, $i
= 1, 2$ be quotient maps. If $E_1$ and $E_2$ are closed subspaces of
$Y_1$ and $Y_2$, respectively, then
\[
(\varphi_1 \obp \varphi_2)^{-1} \Big(\ol{E_1 \ot E_2 }\Big) =\ol{
  \ker(\varphi_1)\ot X_2 + \varphi_1^{-1}(E_1) \ot \varphi_2^{-1}(E_2)
  + X_1 \ot \ker(\varphi_2)}.
  \]
In particular, we have
\[
 \ker (\varphi_1 \obp \varphi_2) = \ol{\ker(\varphi_1)\ot X_2 + X_1 \ot
  \ker(\varphi_2)}.
  \]
 \end{prop}

 \begin{proof}
 Set $Z = \ol{ \ker(\varphi_1)\ot X_2 + \varphi_1^{-1}(E_1) \ot
   \varphi_2^{-1}(E_2) + X_1 \ot \ker(\varphi_2)}$.  Recall that, for
 any subspace $W$ of a Banach space $X$, $W^{\perp \perp} = \ol{W}$,
 where $W^\perp := \{ \Phi \in X^*: \Phi(W) = (0)\} $ (Bipolar
 Theorem). So, it suffices to show that $\Big( (\varphi_1 \obp
 \varphi_2)^{-1} \big(\, \ol{E_1 \ot E_2 }\, \big) \Big)^{\perp}= Z^{\perp}$.

Clearly, $ \Big( (\varphi_1 \obp \varphi_2)^{-1} \big(\, \ol{E_1 \ot
  E_2 }\, \big) \Big)^{\perp} \subseteq Z^{\perp} $. For the reverse
inclusion, let $f \in Z^\perp$. Since $(X_1 \obp X_2)^*$ can be
identified with the space of bounded bilinear forms on $X_1 \times
X_2$ (\cite[$\S 2.2$]{ryan}), there exists a bounded bilinear map
$\tilde{f}: X_1 \times X_2 \to \mathbb{C}$ such that $f(a_1 \ot a_2) =
\tilde{f}(a_1,a_2)$ for $a_i \in X_i$. Define $g: Y_1 \times Y_2 \to
\mathbb{C}$ as $g(b_1,b_2) = \tilde{f}(a_1,a_2)$, where
$\varphi_i(a_i) = b_i, i =1,2$. Since $f|_Z =0$, $g$ is well defined,
and it is also a bounded bilinear map. Thus $g$ can be identified with
a unique element in $ (Y_1 \obp Y_2)^*$, say, $\tilde{g}$. It can be
seen that $ f= \tilde{g} \circ (\varphi_1 \obp \varphi_2).$

 Now, take any $x$ in $(\varphi_1 \obp \varphi_2)^{-1}(\, \ol{E_1 \ot
 E_2}\, )$. Then, by \cite[Proposition 2.8]{ryan}, there exist bounded
 sequences $\{r_n\}$ and $\{s_n\}$ in $E_1$ and $E_2$, respectively,
 such that $$
 (\varphi_1 \obp \varphi_2)(x) = \sum_{n=1}^{\infty} r_n \otimes
 s_n.$$ By surjectivity of $\varphi_i$'s, for each $n \in \N$, fix
 $x_n \in \varphi_1^{-1}(r_n)$ and $y_n \in
 \varphi_2^{-1}(s_n)$. Then,
 $$f(x) = \tilde{g}\Big(\sum_n r_n \ot s_n\Big) = \sum_n
 \tilde{g}\big(r_n \ot s_n\big) = \sum_n f( x_n \ot y_n) = 0, $$ which
 implies that $ Z^{\perp} \subseteq
\Big( (\varphi_1 \obp
  \varphi_2)^{-1} \big(\, \ol{E_1 \ot E_2 }\, \big) \Big)^{\perp}
  $.
  \end{proof}
  
We'll have instances ahead to appeal to the following useful
consequence, wherein  $i_h$ is as in \Cref{obp-min}.

\begin{cor}\label{h-gamma}
  Let $I$ and $J$ be closed ideals in  $C^*$-algebras $A$ and $B$,
  respectively. Then, we have
  \[
  i_h(A \obp J + I \obp B)  =  (A \oh J + I \oh B) \cap i_h(A \obp B).
  \]
  \end{cor}
\begin{proof}
By \Cref{kernel-phi-ot-psi}, we have $\ker(\pi_I \obp \pi_J) = A \obp J +
I \obp B$ and also, by \cite[Corollary 2.6]{ass}, $\ker(\pi_I \oh
\pi_J) = A \oh J + I \oh B$. Now, the  diagram \[
\xymatrix{ A
  \obp B \ar@{^{(}->}[r]^{i_h} \ar[d]^{\pi_I \obp \pi_J} & A \oh B\ar[d]^{\pi_I \oh
    \pi_J} \\ A/I \obp B/J \ar@{^{(}->}[r]^{i_h} & A/I \oh B/J  }
\] is easily
seen to be commutative and we are done.
    \end{proof}

The following folklore result will be required in the proof of
\Cref{obp-ideal}. (Note that, a part of it also follows from
\Cref{kernel-phi-ot-psi}.)

 \begin{lemma}\label{obp-kernel}
Let $A$ and $B$ be $C^*$-algebras, $J$ be a closed ideal in $B$ and
$\pi : B \ra B/J$ be the natural quotient map. Then, $\ker
(\mathrm{Id}\obp \pi) = A \obp J$ and $ (\mathrm{Id}\obp \pi) (Z)$ is
closed for any closed subspace $Z$ in $A \obp B$. Additionally, if $Z$
contains $A \obp J$, then $$ Z = (\mathrm{Id}\obp
\pi)^{-1}((\mathrm{Id}\obp \pi)(Z)). $$
   \end{lemma}
 \begin{proof}
 Clearly, $A \obp J \subseteq \ker (\mathrm{Id}\obp \pi)$, where
the inclusion is meaningful because of \Cref{obp-injective}. Let $z \in
 \ker (\mathrm{Id}\obp \pi)$ and $\epsilon > 0$. Then, by
 \cite[Proposition 2.8]{ryan}, there exist bounded sequences $\{ a_n\}
 \subset A$ and $\{b_n +J \} \subset B/J$ such that $ \sum_n a_n \ot
 (b_n + J) = 0 + J$ and $\sum_n \|a_n\| \|b_n + J\| <
 \epsilon$. Choose a sequence $\{ x_n\} \subset J$ such that $\sum_n
 \|a_n\| \|b_n - x_n \| < \epsilon$. Then, it is easily seen that
 $\sum_n a_n \ot x_n$ is absosutely convergent in $A \obp I$ and that
 $\|\sum_n a_n \ot b_n - \sum_n a_n \ot x_n \|_{\gamma} < \epsilon$,
 which implies that $A \obp I$ is dense in $ \ker (\mathrm{Id}\obp
 \pi)$.
\vspace*{1mm}

Finally, by \cite[Proposition 2.3]{ryan}, $\|\mathrm{Id}\obp \pi \| = \|
\mathrm{Id}\|\, \| \pi\| = 1$ so that $\mathrm{Id}\obp \pi $ is a
contraction and hence $(\mathrm{Id}\obp \pi) (Z)$ is closed.   \end{proof}
 
 Note that \Cref{obp-kernel} holds for Banach spaces and closed
 subspace as well. The same proof works in this generality.

With this, all ingredients are available to adapt the steps of
\cite[Theorem 3.1]{GJ} to obtain a proof of the following ideal
structure:
 
 \begin{theorem}\label{obp-ideal}
 Let $A$ and $B$ be $C^*$-algebras where $A$ is topologically
 simple. Then every closed ideal of the Banach $*$-algebra $A \obp B$
 is a product ideal of the form $ A \obp J$ for some closed ideal $J$
 in $B$. In particular, every closed ideal in $A \obp B$ is a
 $*$-ideal.
   \end{theorem}

 \begin{proof}
We will appeal to the usual Zorn's Lemma approach. Let $I$ be a
non-zero closed ideal in $A \obp B$.  Consider the collection $$
\mcal{F} := \{ J \subseteq B: J
\ \mathrm{is\ a\ closed\ ideal\ in}\ B\ \mathrm{and}\ A \obp J
\subseteq I\},$$ where the inclusion $A \obp J \subseteq I$ makes
sense by \Cref{obp-injective}.  By \Cref{elementary-tensor}, $I$
contains a non-zero elementary tensor, say, $a \ot b$. If $K$ and $J$
are the non-zero closed ideals in $A$ and $B$ generated by $a$ and
$b$, respectively, then by simplicity of $A$, we have $K = A$ and $A
\obp J \subseteq I$. In particular, $\mcal{F} \neq \emptyset$.

We saw in \Cref{finite-sum} that $A \obp (\sum_i J_i) = \sum_i (A \obp
J_i)$ for any finite collection of closed ideals $\{J_i\}$ in $B$.
So, with respect to the partial order given by set inclusion, every
chain $\{ J_i : i \in \Lambda\}$ in $\mcal{F}$ has an upper bound in
$\mcal{F}$, namely, the closure of the ideal
$\big\{\sum_{\mathrm{finite}} x_i : x_i \in J_i\big\}$, implying
thereby that there exists a maximal element, say $J$, in $\mcal{F}$.

The obvious thing to do now is to try to show that $A \obp J =
I$. Consider the canonical map $\mathrm{Id} \obp \pi : A \obp B \ra A
\obp (B / J)$. By \Cref{obp-kernel}, we have $\ker(\mathrm{Id}\obp
\pi) = A \obp J$ and that $\widetilde{I}:=(\mathrm{Id} \obp \pi)(I)$
is a closed ideal in $A \obp (B / J)$. It now suffices to show that
$\widetilde{I} = (0)$. If $\widetilde{I} \neq (0)$, then, again by
\Cref{elementary-tensor}, $\widetilde{I}$ contains a non-zero
elementary tensor, say, $a \ot (b +J)$.  Observe that $b \notin J$
and $$a \ot b \in (\mathrm{Id}\obp \pi)^{-1}(a \ot (b + J)) \in
(\mathrm{Id}\obp \pi)^{-1} \big ((\mathrm{Id}\obp \pi) (I)\big) = I $$
by \Cref{obp-kernel}. Let $K$ denote the closed ideal in $B$ generated
by $b$ and $J$. Note that $J \subsetneq K$. Since $A$ is simple, it
equals the closed ideal generated by $a$ and we obtain $A \obp K
\subseteq I$, i.e., $K \in \mcal{F}$ contradicting the maximality of
$J$ in $\mcal{F}$. 
\end{proof}

 \noindent In view of Theorems \ref{obp-injective}, \ref{obp-ideal}
 and \Cref{obp-kernel}, we obtain the following analogue of
 \cite[Corollary 4.21]{Tak}, \cite[Theorem 5.1]{ass} and of
 \cite[Theorem 3.7]{JK-edin}.

 \begin{cor}\label{obp-simple}
 Let $A$ and $B$ be  $C^*$-algebras. Then the
 Banach $*$-algebra $A \obp B$ is topologically simple if and only if
 $A$ and $B$ are both topologically simple.
   \end{cor}

If $A$ contains only finitely many closed ideals and $\alpha$ is
either the Haagerup or the operator space projective tensor product,
then every closed ideal in the Banach algebra $A \ot^{\alpha} B$ is a
finite sum of product ideals, - see \cite{ass, kr-2}.  We now make
another use of \Cref{obp-ideal} to prove its analogue for $A \obp
B$. We'll use the following useful observation made in the proof of
\cite[Theorem 5.3]{ass}.
\begin{lemma}\label{ann-result}
  Let $A$ and $B$ be $C^*$-algebras and $I$ be a  simple closed  ideal
  in $A$.  If $K$ is a closed ideal in the Banach algebra $A \oh B$, then $K \cap (I \oh
  B) = I \oh J$ for some closed ideal $J$ in $ B$, and, 
  \[
K \subseteq A \oh J +  M \oh B,
\]
where $M$ is the closed ideal $ann(I) :=\{ x \in A: x
  I = I x = (0)\}$.
  \end{lemma}

 \begin{theorem}\label{ideals-finite}
Let $A$ and $B$ be $C^*$-algebras and suppose $A$ contains only
finitely many closed ideals. Then every closed ideal in the Banach
$*$-algebra $A \obp B$ is a finite sum of product ideals.
 \end{theorem}

 \begin{proof}
The proof is given by induction on the number of closed ideals in $A$,
call it $\nu(A)$. If $\nu(A) = 2$, then we are done by
\Cref{obp-ideal}. Let $n > 2$ and suppose that the assertion holds
for all $C^*$-algebras with less than $ n$ closed ideals, and let $A$
be a $C^*$-algebra with $\nu(A) = n$.

 Pick  a minimal non-zero closed ideal, say $I$, in $A$, which is
 clearly simple.  Let $K$ be a closed ideal in $A \obp B$, then $I
 \obp B \subseteq A \obp B$ by \Cref{obp-injective}, so that $K\cap (I
 \obp B)$ is a closed ideal in $I \obp B$. By \Cref{obp-ideal}, it is then
 equal to $I \obp J$ for some closed ideal $J$ in $B$.

Consider the closed ideal $K_h:=\overline{i_h(K)}^h$ in the Banach
algebra $ A \oh B$ (where $i_h$ is as in \Cref{obp-min}). By
\Cref{ann-result}, $K_h \cap (I \oh B) = I \oh J_1$ and $K_h \subseteq
A \oh J_1 + M\oh B$ for some closed ideal $J_1$ in $B$, where $M =
ann(I)$.

We wish to show, in fact, that $ K \subseteq A \obp J + M \obp B$ as
well. Note that $K \subseteq i_h^{-1}(K_h) \subseteq i_h^{-1}(A \oh
J_1 + M \oh B)$ and, by \Cref{h-gamma}, $i_h^{-1}(A \oh J_1 + M \oh B)
= A \obp J_1 + M \obp B$. So, it suffices to show
that $J_1 \subset J$. Note that, if $y \in J_1$, then $x \ot y \in I
\oh J_1 \subset K_h $ for any fixed $0 \neq x \in I$, so that, by
\Cref{a-ot-b-h}, $x \ot y \in K$ implying further that $ x \ot y \in K
\cap (I \obp B) = I \obp J$. Choose a $\varphi \in A^*$ such that
$\varphi(x) \neq 0$, then $R_\varphi(x\ot y) = \varphi(x) y \in J$,
where $R_{\varphi}:A \otimes B \rightarrow B$ is the right slice map
given by $ R_{\varphi}\left(\sum_1^n a_i \otimes
b_i\right)=\sum_{1}^{n} \varphi (a_i) b_i$. Hence $y \in J$.

 We now claim that $K \cap (A \obp J +M \obp B)=K\cap
 (A \obp J)+ K \cap (M \obp B)$ and show that each of
 the two closed ideals appearing in the sum on the right hand side, by induction
 hypothesis, are finite sums of closed ideals, which will then complete the proof.

 We first prove that $L := K \cap (A \obp J)$ is a finite sum of closed ideals. 
 Clearly $L$ contains $I\obp J$. 
 
 Corresponding to the complete quotient map $\pi_I: A \to A/I$, we
 have a quotient map $\pi_I \obp Id: A \obp J \to A/I \obp J$. By
 \Cref{obp-kernel}, $\ker (\pi_I \obp Id) = I \obp J$ and $(\pi_I \obp Id)(L)$
 is a closed ideal in $A/I \obp J$. Since $\nu (A/I )\leq n-1$, by
 induction hypothesis, $(\pi_I \obp Id)(L)= \sum_{r=1}^k I_r \obp J_r$,
 where $I_r$ and $J_r$ are closed ideals in $A/I $ and $J$,
 respectively. Thus, by \Cref{obp-kernel}  and \Cref{kernel-phi-ot-psi},
 $ L = \sum_{r=1}^k \pi_I^{-1}(I_r)\obp J_r +  I \obp J$.
 
Further, since $M$ cannot contain $I$, we have $\nu(M)\leq n - 1$.  So,
by induction hypothesis, the closed ideal $K \cap (M \obp B)$ is a
finite sum of product ideals.

Finally, it is easy to see that $ K\cap (A \obp J)+ K \cap (M \obp B)
\subseteq K \cap (A \obp J + M \obp B) $.  Let $z\in K \cap (A \obp J
+ M \obp B)$. By \cite[Proposition 4.11]{GJ}, the closed ideal $A \obp
J + M \obp B$ possesses a quasi-central approximate identity, say $\{
e_{\lambda}\}$, which as in \cite[Lemma 3.3]{ass}, can be taken to be
of the form $e_{\lambda} = f_{\lambda} + g_{\lambda} - f_{\lambda}
g_{\lambda}$ for some quasi-central approximate identities
$\{f_{\lambda}\}$ and $\{g_{\lambda}\}$ in $A \obp J$ and $M \obp B$,
respectively. Then, $e_{\lambda} z \ra z$ and $ze_{\lambda} \in K\cap
(A \obp J)+ K \cap (M \obp B)$ for every $\lambda$. This implies that
$K\cap (A \obp J)+ K \cap (M \obp B)$ is dense in $K \cap (A \obp J +
M \obp B) $ and, by \Cref{sum-obp}, we know that $K\cap (A \obp J)+ K
\cap (M \obp B)$, being a finite sum of product ideals, is closed.
\end{proof}

Based on above discussion, we easily deduce the following: 
\begin{cor}\label{consequences}
Let $A$ and $B$ be $C^*$-algebras and suppose $A$ contains only
finitely many closed ideals. Then, the following hold:
\begin{enumerate}
\item A finite sum of closed ideals in $A
\obp B$ is also a closed ideal.
\item Every closed ideal of $A \obp B$ contains a bounded approximate unit.
  \item Every closed ideal of $A \obp B$ is $*$-closed.
\end{enumerate}
\end{cor}

\subsection*{Some Examples}
We can now reap some immediate fruits of \Cref{ideals-finite} and
\Cref{consequences}.
\begin{enumerate}
  \item For any separable Hilbert space $H$, analogous to the
    structure of closed ideals of $B(H) \ota B(H)$ for $\ota = \oh$
    and $\oop$ (see \cite{ass,jk11}), $B(H)\obp B(H)$ contains only
    four non-trivial closed ideals, namely, $ B(H) \obp K(H) + K(H)
    \obp B(H), B(H) \obp K(H), K(H) \obp B(H)$ and $ K(H) \obp K(H)$.
    In particular, $B(H) \obp B(H)$ has a unique maximal ideal,
    namely, $B(H) \obp K(H) + K(H) \obp B(H) $.

\item For a locally compact Hausdorff space $X$ and a separable
  Hilbert space $H$, the closed ideals of $B(H) \obp C_0(X)$ are given
  by $\sum_{i=1}^n I_i\obp I(E_i)$, for closed subsets $E_i$ of $X$,
  where $I_i=K(H) $ or $B(H)$, and $$I(E_i) := \{ f\in C_0(X): f(x)=0
  \ \text{for all} \,x\in E_i\}.$$

  Note that, by \Cref{maximal-ideals}, $B(H) \obp I(\{x\}) + K(H) \obp
  C_0(X)$ is a maximal ideal in $B(H) \obp C_0(X)$ for each $x \in
  C_0(X)$.
  \end{enumerate}

\subsection{Minimal and maximal ideals of $A \obp B$}

The structure of closed minimal ideals of $A \obp
B$ turns out to be an immediate consequence of \Cref{elementary-tensor}.
\begin{prop}
  Let $A$ and $B$ be $C^*$-algebras. Then, a closed ideal $J$ in $A
  \obp B$ is minimal if and only if it is a product ideal of the form
  $J = K \obp L$ for some minimal closed ideals $K$ and $L$ in $A$ and
  $B$, respectively.
\end{prop}
    The proof of \cite[Proposition 3.9]{JK-edin} works verbatim
    for above identification.
    \vspace*{2mm}
    
In order to analyze the structure of maximal ideals, we will use the
concept of {\em Wiener property} for Banach $*$-algebras.  Recall that
a Banach $*$-algebra $A$ said to have the Wiener property if every proper
closed ideal of $A$ is annihilated by some irreducible
$*$-representation of $A$ on some Hilbert space.  It is well known
that every $C^*$-algebra has Wiener property.

\begin{lemma}\label{wiener}
  Let $A$ and $B$ be $C^*$-algebras. Then, the Banach $\ast$-algebra
  $A\obp B$ has the Wiener property.
\end{lemma}

\begin{proof}
Consider a proper closed two-sided ideal $J$ of $A\obp B$. By Theorem
\ref{a-ot-b}, $J_{\min}$ is also a proper closed two-sided ideal of
the $C^*$-algebra $A\omin B$. Since every $C^*$-algebra has the Wiener
property, $J_{\min}$ is annihilated by an irreducible
$*$-representation, say, $\pi:A\omin B \rightarrow B(H)$.  So, we have
a $*$-representation $\hat{\pi}:= \pi \circ i$ of $A\obp B$ on $H$
with $\hat{\pi}(J)=\{0\}$, where $i : A \obp B \ra A \omin B$ is the
canonical injective $*$-homomorphism as in \Cref{obp-min}. Also, the
relation $\hat{\pi}(A\otimes B) = \pi (A\otimes B)$ yields
\[
\hat{\pi}(A\obp B)^{\prime} \subseteq \hat{\pi}(A\otimes B)' =
\pi(A\otimes B)^{\prime} = \pi(A\omin B)^{\prime}= \mathbb{C}I.
\] Thus, $\hat{\pi}$ is irreducible and $A\obp B$ has
Wiener property.
\end{proof}

\begin{lemma}\label{irr}
  Let $A$ and $B$ be $C^*$-algebras and $\pi$ be an irreducible
  $*$-representation of $A\obp B$ on a Hilbert space $H$. Then there
  exist $*$-representations $\pi_1$ and $\pi_2$ of $A$ and $B$,
  respectively, on $H$ with commuting ranges such that
$$ \pi(a\ot b) = \pi_1(a) \pi_2(b) \,\,\, \text{for all} \,\, a\in A, b\in B.$$
Moreover, $\pi_1$ and $\pi_2$ are both factor representations. 
\end{lemma}

\begin{proof}
Since $\pi$ is a $*$-representation of $A\ot B$, by \cite[Lemma
  IV.4.1]{Tak}, there exist $*$-representations $\pi_1$ and $\pi_2$ of
$A$ and $B$ on $H$ with commuting ranges such that
$$ \pi(a\otimes b) = \pi_1 (a) \pi_2 (b)\,\, \text{for all} \,a\in
A,\,b\in B. $$

Now, $\pi(A\otimes B)= \pi_1(A) \pi_2(B)$, so that $ \pi(A\obp B)
\subseteq \ol{\pi_1(A) \pi_2(B)}$. Irreducibility of $\pi$ gives
$$ \big(\pi_1(A) \pi_2(B)\big)^{\prime}=\Big( \ol{\pi_1(A)
  \pi_2(B)}\Big)^{\prime} \subseteq \pi(A \obp B)^{\prime} =
\mathbb{C} I.
$$ If $ M := \pi_1(A)^{\prime\prime}$, then we have \begin{eqnarray*}
  M \cap M^{\prime} & = & \pi_1(A)^{\prime\prime} \cap
  \pi_1(A)^{\prime}\\
  &=& \big(\pi_1(A)^{\prime} \cup
  \pi_1(A)\big)^{\prime} \\
  &\subseteq & (\pi_2(B) \cup
  \pi_1(A))^{\prime} \, \qquad \qquad ( \text{as} \, \pi_2(B) \subseteq
  \pi_1(A)')\\
  & = &\pi_1(A)^{\prime} \cap
  \pi_2(B)^{\prime}\\
  & \subseteq & \{\pi_1(A)
  \pi_2(B) \}^{\prime}\\ &=& \mathbb{C} I.
\end{eqnarray*}
Thus, $\pi_1$ (and similarly $\pi_2$) is a factor representation. \end{proof}

Analogous to \cite[Theorem 5.6]{ass},   \cite[Theorem
  3.10]{JK-edin} and \cite[Theorem 9]{jk11}, we  now obtain the
following characterizations of maximal and maximal modular
ideals.

 \begin{theorem}\label{maximal-ideals}
 Let $A$ and $B$ be $C^*$-algebras. Then, a closed ideal $J$ in $A
 \obp B$ is maximal if and only if it is of the form $J = A \obp N + M
 \obp B$ for some maximal  ideals $M$ and $N$ in $A$ and $B$,
 respectively.
   \end{theorem}
   
   \begin{proof}
    Let $J = A \obp N + M \obp B$, where $M$ and $N$ are maximal
    ideals in $A$ and $B$, respectively. Note that, by \Cref{sum-obp},
    $A\obp N + M\obp B$ is a closed ideal in $A\obp B$. For the
    canonical quotient maps $\pi_1 : A \ra A/M$ and $\pi_2 : B \ra
    B/N$, we have $ J =\text{ker}(\pi_1 \otimes \pi_2)$ and there is
    an isomorphism between $(A\obp B)/J $ and $(A/M)\obp (B/N)$ by
    \Cref{kernel-phi-ot-psi}. Since $A/M$ and $B/N$ are both simple,
    so is $(A\obp B)/J$, by \Cref{obp-simple}. Thus, $J$ is maximal in
    $A\obp B$.

Conversely, let $J$ be a maximal ideal in $A\obp B$. As seen in
\Cref{wiener}, $A\obp B$ has the Wiener property; so, there exists a
non-zero irreducible $*$-representation $\pi$ of $A\obp B$ on a
Hilbert space $H$, such that $\pi(J)= (0)$. Then, by \Cref{irr}, there
exist factor $*$-representations $\pi_1$ and $\pi_2$ of $A$ and $B$ on
$H$ with commuting ranges such that $ \pi(a\otimes b) = \pi_1 (a)
\pi_2 (b)$ for all $a\in A$ and $b\in B $. Set $M= \ker \pi_1,\,N=
\ker \pi_2$ and $L= A\obp N + M\obp B$. We first show that $ L = J$.

 Clearly, $\pi(M\obp B)= (0) = \pi(A\obp N)$, which gives $\pi(J+L)=
 (0)$. Since $\pi$ is non-zero, this shows that $J+L$ is a proper
 ideal of $A\obp B$. Since $J + L$ contains $J$, by maximality of $J$,
 we have $L \subseteq J$, i.e., $ A\obp N + M \obp B \subseteq J.$

 By \Cref{kernel-phi-ot-psi} and \Cref{sum-obp} we have $\ker(\pi_M
 \obp \pi_N) = A \obp N + M \obp B$, so, for the reverse inclusion, it
 suffices to show that $J \subseteq \ker(\pi_M \obp \pi_N)$, where
 $\pi_M$ and $\pi_N$ are the natural quotient maps.  Note that the
 representations $\pi_1$ and $\pi_2$ induce faithful commuting
 representations $\tilde{\pi}_1$ of $A/M$ and $\tilde{\pi}_2$ of $B/N$
 on $H$.  Then, by a universal property of $\ot^{\max}$ (see
 \cite[Proposition IV.4.7]{Tak}), there exists a bounded
 $*$-representation $\pi_0:(A/M) \ot^{\max} (B/N) \ra B(H)$ such that
 $\pi_0(x\otimes y) = \tilde{\pi}_1 (x)\tilde{\pi}_2(y)$ for all $x\in
 A/M$ and $ y\in B/N $. Since $\|\cdot \|_{\max} \leq \| \cdot
 \|_{\gamma}$, the identity map on $(A/M) \ot (B/N)$ extends to a
 contractive $*$-homomorphism, say, $i: (A/M) \obp (B/N) \ra (A/M)
 \otimes^{\max} (B/N)$. In particular, $\theta:= \pi_0 \circ i$ is a
 $*$-representation of $ (A/M) \obp (B/N)$ on $H$.

It is easy to verify that $ \pi = \theta \circ \, (\pi_M \obp \pi_N) $
on $A\otimes B$, so by continuity we have $ \pi = \theta \circ \,
(\pi_M \obp \pi_N) $, which further gives $ \theta ( (\pi_M \obp
\pi_N)(J))=0 $. We now claim that $\theta$ is faithful on $(A/M) \obp
(B/N)$, which will yield $(\pi_M \obp \pi_N)(J)= (0)$, as was asserted
above. Note that, by \Cref{obp-faithful}, it suffices to show that
$\theta$ is faithful on $(A/M) \otimes (B/N)$. Since $\pi_1$ and
$\pi_2$ are both factor representations, so are the representations
$\tilde{\pi}_1$ and $\tilde{\pi}_2$ because
\[
\tilde{\pi}_1(A/M)^{\prime\prime} = \pi_1(A)^{\prime\prime} \quad \text{and}\quad
\tilde{\pi}_2(B/N)^{\prime\prime} = \pi_2(B)^{\prime\prime}.
\]
Now, for the factor $\mathcal{R}= \tilde{\pi}_1(A/M)^{\prime \prime}$,
the map $$\mathcal{R} \otimes \mathcal{R}^{\prime} \ni \Sigma_{i=1}^n
x_i \otimes x_i^{\prime} \stackrel{\rho}{\mapsto} \Sigma_{i=1}^n x_i
x_i^{\prime} \in B(H)$$ is an injective homomorphism, by
\cite[Proposition IV.4.20]{Tak}. Suppose $ \theta\big(\sum_{i=1}^n x_i
\otimes y_i\big) = 0$ for some $\sum_{i=1}^n x_i
\otimes y_i \in (A/M) \ot (B/N)$, which gives
$$
0 = \theta\Big(\sum_i x_i \otimes y_i\Big) = \sum_i
\tilde{\pi}_1(x_i) \tilde{\pi}_2 (y_i) = \rho\Big( \sum_i
\tilde{\pi}_1(x_i) \otimes \tilde{\pi}_2 (y_i)\Big).
$$ Note that $\tilde{\pi}_2(y_i) \in \tilde{\pi}_1(A/M)' =
\mcal{R}'''= \mcal{R}' $ for all $i$. Since $\rho$ is injective, we
obtain $$
(\tilde{\pi}_1 \ot \tilde{\pi}_2) \Big(\sum_i x_i \otimes y_i\Big) =
\sum_i \tilde{\pi}_1(x_i) \otimes \tilde{\pi}_2 (y_i) = 0.
$$ Further,
since $ \tilde{\pi}_1$ and $ \tilde{\pi}_2$ are both injective, so is $
\tilde{\pi}_1 \otimes \tilde{\pi}_2$ and hence $\sum_i x_i \otimes y_i
=0$.  This proves our claim.

Finally, since $J = \ker(\pi_M \obp \pi_N)$, $(A\obp B)/J$ is
isomorphic to $(A/M) \obp (B/N)$.  So, by \Cref{obp-simple}, it
follows that $M$ and $N$ are both maximal in $A$ and $B$,
respectively.
   \end{proof}

It is known that an ideal in a Banach algebra is maximal modular if and
only if it is maximal and modular. The above structure of maximal
ideals immediately yields the structure of maximal modular ideals as
well.
\begin{theorem}\label{maximal-modular}
  Let $A$ and $B$ be $C^*$-algebras. Then, a closed ideal $J$ of
  $A\obp B$ is maximal modular if and only if it is of the form $J = A
  \obp N + M \obp B$ for some maximal modular ideals $M$ and $N$ in
  $A$ and $B$, respectively.
   \end{theorem}
   \begin{proof}
Let $J = A\obp N + M \obp B$ for some maximal modular ideals $M$ and
$N$ in $A$ and $B$, respectively. Since $M$ and $N$ are both maximal
ideals, so is $J$, by \Cref{maximal-ideals}. Also, by
\Cref{kernel-phi-ot-psi}, $(A\obp B)/J$ and $\big(A/M\big) \obp
\big(B/N\big)$ are isomorphic Banach $*$-algebras. Since $A/M$ and
$B/N$ are both unital, so is $(A\obp B)/J$. In particular, $J$ is
modular.

Conversely, suppose $J$ is a maximal modular ideal in $A \obp
B$. Again, by \Cref{maximal-ideals}, $J$, being maximal, is of the
form $J = A\obp M + I\obp N$ for some maximal ideals $M$ and $N$ in
$A$ and $B$, respectively. As seen in previous paragraph, $(A\obp
B)/J$ is isomorphic to $\big(A/ M \big) \obp\big( B/N\big)$; in
particular, the latter space is unital. Therefore, by \cite[Theorem
  1]{loy}, $A/M$ and $B/N$ are both unital; so that $M$ and $N$ are
both modular as well.  \end{proof}
   
   \section{Hull-kernel topology}

As in Introduction, for any Banach algebra $A$, 
$Id(A)$ (resp., $Id'(A)$) denotes the set of closed ideals (resp.,
proper closed ideals) of $A$. And, for any algebra $A$,  $\m(A)$ (resp.,  
$\m_m(A)$) denotes the set of maximal (resp., maximal modular) ideals of $A$.

Before discussing hull-kernel topology, we briefly outline another
topology on $Id(A)$ for a Banach algebra $A$, which agrees with
hull-kernel topology on the set of maximal ideals and is called the {\it
  $\tau_w$-topology} (\cite[$\S 2$]{arc2}). A subbasis for
$\tau_w$-topology is given by the collection
$$\Big\{ U(J) :=\{ I \in Id(A): I\nsupseteq J\},\, J \in Id(A) \cup
\{\emptyset\} \Big\},$$ where $U(\emptyset) := A $. Note that $U(A) =
Id'(A)$, $U((0)) = \emptyset$ and $U(\emptyset)$ is the only subbasic
set that contains $A$. $Id(A)$ is a $T_0$ space with respect to
$\tau_w$-topology (\cite{arc2}).

 For $C^*$-algebras $A$ and $B$, consider the map $\Phi: Id(A) \times Id(B)
\to Id(A\obp B)$ defined by
\begin{equation}\label{Phi} \Phi(I, J) = 
  \ker(\pi_I \obp \pi_J) = A\obp J + I \obp B,
\end{equation}
where $\pi_I: A \to A/I$ and $\pi_J: B\to B/J$ are the canonical
quotient maps. Note that the last equality in (\ref{Phi}) follows from
\Cref{kernel-phi-ot-psi} and \Cref{sum-obp}. Also, If $(I, J) \in
Id'(A) \times Id'(B)$, then $\pi_I \obp \pi_J \neq 0$ so that $
\ker(\pi_I \obp \pi_J) $ is proper. Hence $\Phi$ maps $Id'(A) \times
Id'(B)$ into $ Id'(A\obp B)$. Analogous to \cite[Lemma 1.4]{arc3} and \cite[Lemma 2.5]{laz}, we
obtain the following:

\begin{lemma}
 Let $A$ and $B$ be $C^*$-algebras. Then, $\Phi: Id(A) \times Id(B) \ra
Id (A \obp B)$ is $\tau_w$-continuous.
  \end{lemma}
  
\begin{proof}
Consider the diagram
\[
\xymatrix {  & Id(A) \times Id(B)  \ar[ld]_{\Phi_1} \ar[rd]^{\Phi} & \\
Id(A\omin B)  \ar[rr]^{\Phi_2} & & Id(A \obp B),}
\] 
where $\Phi_1(I,J) := \ker(\pi_I \omin \pi_J)$ and $\Phi_2(K) :=
i^{-1}(K) $, $i$ being the injective contractive $*$-homomorphism from
$A\obp B \to A\omin B$ (as in \Cref{obp-min}). It is known that
$\Phi_1$ is $\tau_w$-continuous - see \cite[Lemma 2.5]{laz}. So, it
suffices to show that this diagram commutes and that $\Phi_2$ is
$\tau_w$-continuous.

In order to establish commutativity of the diagram, we just need to
verify that
$$ \ker(\pi_I \obp \pi_J) = i^{-1}(\ker(\pi_I \omin \pi_J)). $$ For
$z\in A\obp B$, let $\{z_n\}$ be a sequence in $A\ot B$ such that $
\|z_n -z\|_{\gamma} \ra 0$. Let $\hat{i}: (A/I) \obp (B/J) \to (A/I)
\omin (B/J)$ be the injective continuous homomorphism. Then the
sequence $ \{\hat{i}\big((\pi_I \obp \pi_J)(z_n)\big) = (\pi_I \obp
\pi_J)(z_n)\} $ converges to $\hat{i}\big((\pi_I \obp \pi_J)(z)\big)$
in $(A/I) \omin (B/J)$.  Since $\|\cdot \|_{\min} \leq \| \cdot
\|_{\gamma}$ and $i$ is identity on $A \ot B$, $ \|z_n -i(z)\|_{\min}
\ra 0$ as well. So, $ (\pi_I \omin \pi_J)(z_n) {\longrightarrow}
(\pi_I \omin \pi_J)(i(z))$ in $(A/I) \omin (B/J)$. Since both the
mappings $\pi_I \obp \pi_J$ and $\pi_I \omin \pi_J$ agree on $A\ot B$,
by continuity, we have $$\hat{i}\big((\pi_I \obp \pi_J)(z)\big) =
(\pi_I \omin \pi_J)(i(z)).$$ The required relationship now follows from
injectivity of $\hat{i}$.

Next, we show  $ \Phi_2$ is $\tau_w$-continuous. For a subbasic
open set $U(K)$ of $Id(A\obp B)$ for some $K\in Id(A\obp B)$, we have
$\Phi_2^{-1} (U(K)) = U(K_{\min})$. Indeed, for $P \in Id(A\omin
B)$,
\begin{eqnarray*}
P \in \Phi_2^{-1} (U(K))  & \iff & i^{-1}(P) \in U(K) \\
&   \iff & i^{-1}(P)  \nsupseteq K \\
&  \iff & P \nsupseteq K_{\min} \quad (:=\ol{i(K)}) \\
& \iff & P \in U(K_{\min}).
\end{eqnarray*}
Thus, $ \Phi_2$ is $\tau_w$-continuous.  \end{proof}

\begin{lemma}{\label{ker-cont}}
 Let $A$ and $B$ be $C^*$-algebras and $(I_i, J_i) \in Id'(A)\times
 Id'(B),\, i=1,2$ be such that $ A\obp J_1 + I_1 \obp B \seq A\obp J_2
 + I_2 \obp B $. Then $I_1 \seq I_2$ and $J_1 \seq J_2$.
\end{lemma}

\begin{proof}
 For a fixed $a\in I_1$ and any $b \in B$ we have $a \otimes b \in
 \ker (\pi_{I_1}\obp \pi_{J_1}) \seq \ker (\pi_{I_2}\obp \pi_{J_2})$,
 by the given condition. This yields $\pi_{I_2} (a) \otimes
 \pi_{J_2}(b) = 0$ for every $b\in B$. Since $J_2$ is proper,
 $\pi_{J_2}(b) \neq 0$ for some $b \in B$. So, we must have
 $\pi_{I_2}(a)=0$, that is, $a\in I_2$. Similarly, we  obtain $J_1
 \subseteq J_2$.
\end{proof}

We now obtain the following
analogue of \cite[Theorem 1.5]{arc3}, \cite[Theorem 2.6]{laz} and
\cite[Proposition 1.1(v)]{jk-13}.
\begin{theorem}
 Let $A$ and $B$ be $C^*$-algebras. Then, $\Phi$ maps $Id'(A) \times
 Id'(B)$ homeomorphically onto its image which is dense in $Id'(A \obp B)$.
\end{theorem}

\begin{proof}
 For a closed ideal $K$ in $A\obp  B$, define 
 $$K_A=\{ a \in A: a\ot B \seq K\}\quad \text{and} \quad K_B = \{ b\in
 B: A \ot b \seq K\}.$$ By an easy application of continuous functional calculus, it is
 immediately seen that $K_A$ and $K_B$ are closed ideals in $A$ and $B$,
 respectively. Define $\Psi: Id(A\obp B) \to Id(A) \times Id(B)$ by $
 \Psi (K) = (K_A, K_B)$. Then $\Psi \circ \Phi$ equals {identity} on
 $Id'(A) \times Id'(B)$. To see this, consider $(I,J) \in Id'(A)
 \times Id'(B)$ and set $K= \Phi(I,J)= A\obp J + I\obp B$. Clearly
 $K_A \supseteq I$ and $K_B \supseteq J$. Also, for $a\in K_A$, if
 $I^\prime$ denotes the closed ideal generated by $a$ in $A$, then
 $I^\prime \obp B \seq K$. So, $I^\prime \in Id'(A)$ and by
 \Cref{ker-cont}, $I^{\prime} \seq I$, giving that $K_A \seq I$ and
 hence $K_A = I$. Similarly, we can see that $K_B = J$. As a
 consequence, $\Phi$ is injective on $Id'(A) \times Id'(B)$.

 It now
 suffices to show that $\Psi$ is $\tau_w$-continuous. Consider a
 subbasic open set $U(I) \times U(J)$ of $Id(A) \times Id(B)$, where
 $I \in Id(A), J \in Id(B)$. For any $K \in Id(A\obp B)$,
 \begin{eqnarray*}
  K \in \Psi^{-1}(U(I) \times U(J)) & \iff & K_A \in U(I) \quad \&
  \quad K_B \in U(J) \\ & \iff & K_A\nsupseteq I \quad \& \quad K_B
  \nsupseteq J \\ & \iff & K \nsupseteq I\obp B \quad \& \quad K
  \nsupseteq A\obp J \\
  & \iff & K \in U(I\obp B) \cap U(A\obp J) 
   \end{eqnarray*}
   Thus $\Psi^{-1} (U(I) \times U(J)) = U(I \obp B) \cap U(A\obp J)$
   and since the latter set is open in $Id(A\obp B)$, this proves our
   claim.

 We now show that $\Phi \big(Id'(A) \times Id'(B)\big)$ is dense
 in $Id'(A \obp B)$. For this, consider a $K$ in $Id'(A \obp B)$ and let
 $U$ be a basic open set in $Id'(A\obp B)$ containing $K$. Then $U =
 \cap_{i=1}^n U(P_i)$ for some $P_i \in Id'(A\obp B)$. Here, $K \in
 U(P_i)$, that is, $P_i \nsubseteq K$ for all $1 \leq i \leq n$. Now
 for each $i$, note that $A\obp K_B + K_A \obp B \seq K$; so
 $P_i\nsubseteq A \obp K_B$ and $ P_i \nsubseteq K_A \obp B$, since
 $P_i \nsubseteq K$.  This further implies that $P_i \nsubseteq
 \Phi(0,K_B)$ and $P_i\nsubseteq \Phi(K_A,0)$ so that $\Phi(0,K_B) \in
 U(P_i)$ and $\Phi(K_A,0) \in U(P_i)$ for all $1 \leq i \leq n$. Thus
 $U \cap Im(\Phi) \neq \phi$ and hence image of $\Phi$ is dense in
 $Id'(A\obp B)$.
\end{proof}

We now briefly recall  hull-kernel topology, without details.  Let
$A$ be a Banach algebra. For each $E \subseteq Prime(A)$, the set of
all proper closed prime ideals of $A$, one associates a closed ideal,
called {\it kernel} of $E$, given by $ k(E) = \bigcap_{P \in E} P.$
Also, for each $M \subseteq A$, {\it hull} of $M$ is defined as
$$ h_A(M) = \{ P \in Prime(A) : P \supseteq M \}.$$ Equip $Prime(A)$
with the {\it hull-kernel topology} (hk-topology, in short), where for
$E\subseteq Prime(A)$, its closure turns out to satisfy $\overline{E}
= h(k(E))$, which can be taken as the definition of closure for our
purpose - for details, see \cite{arc2} and references therein.

As mentioned above, it is a fact that for any Banach algebra $A$, the $\tau_w$-topology
coincides with the hull-kernel topolgy on $\m(A)$- see
\cite{arc2}. The above homeomorphism restricts well to maximal and
maximal modular ideals.

\begin{theorem}
Let $A$ and $B$ be $C^*$-algebras.  Then, the restriction of $\Phi$ to
$\m(A) \times \m(B)$ is a homeomorphism onto $\m(A \obp B)$ with
respect to the hull-kernel topology. Furthermore, $\Phi$ maps
$\m_m(A) \times \m_m(B)$ homeomorphically onto $\m_m(A
\obp B)$, as well.
\end{theorem}

\begin{proof}
By \Cref{Phi}, \Cref{maximal-ideals} and \Cref{maximal-modular}, $\Phi$ maps
$\m(A) \times \m(B)$ (resp., $\m_m(A) \times \m_m(B)$) bijectively onto
$\m(A \obp B)$ (resp., $\m_m(A \obp B)$).

Since $\Phi$ is continuous (being $\tau_w$-continuous), it just remains
to show that $\Phi$ is a closed map with respect to the product
$hk$-topology on $\m(A) \times \m(B)$ and the $hk$-topology on $\m(A
\obp B)$. First of all note that any closed set in $\m(A) \times
\m(B)$ is of the form $$ \cap_\alpha \Big(\big(F_A^\alpha \times
\m(B)\big) \cup \big(\m(A) \times F_B^\alpha\big)\Big),$$ where
$F_C^\alpha$ is closed  in $\m(C)$ for $C = A,
B$. Also, since $\Phi$ is injective, we have
 $$ \Phi \Big(\cap_\alpha \Big( \big(F_A^\alpha \times \m(B)\big) \cup
\big(\m(A) \times F_B^\alpha) \big) \Big) = \cap_\alpha
\Big(\Phi\big(F_A^\alpha \times \m(B)\big) \cup \Phi\big(\m(A) \times
F_B^\alpha\big)\Big).$$
Thus, it is sufficient to prove that for any
closed set $F_A$ of $\m(A)$, $X:=\Phi\big(F_A \times \m(B)\big)$ is
closed in $\m(A \obp B)$. We have to show that $h(k(X)) \subseteq
X$. Let $P\in \m(A \obp B)$ be such that $k(X) \subseteq P$. Let, if
possible, $P\notin X$. Since $P = A\obp J + I \obp B= \ker(\pi_I \obp
\pi_J)$, for some $I \in \m(A), J \in \m(B)$, it follows that $I
\notin F_1$. But $F_A$ is closed in $hk$-topology, thus $k(F_A)
\nsubseteq I$. Let $a \in k(F_A) \setminus I$ and fix $b \notin
J$. Note that $(a+ I) \otimes (b+J) \neq 0$ which gives that $a
\otimes b \notin \ker(\pi_I \obp \pi_J) = P$. On the other hand,
consider any $K:= \Phi(L \times M) \in X$, where $ L \in F_A$ and $ M
\in \m(B)$. Since $ a\in k(F_A) \subseteq L $ we have $ (a+L) \otimes
(b+M) =0$. Thus $a\otimes b \in \ker(\pi_L \obp \pi_M) = K$ and this
is true for all $K \in X$. So, $a\otimes b \in k(X) \subseteq P$,
which gives a contradiction. \end{proof}

\section{Center of $A\obp B$}

For algebras $A$ and $B$ with centers $\mcal{Z}(A)$ and $\mcal{Z}(B)$,
respectively, one can easily check that there is a canonical algebra
isomorphism between $\mcal{Z}(A \ot B)$ and $ \mcal{Z}(A) \ot\mcal{Z}(
B)$. For any two $C^*$-algebras $A$ and $B$ and any $C^*$-norm $\|
\cdot \|_\alpha$, it is known that the above isomorphism extends to
an isometric $*$-isomorphism from $\mcal{Z}(A \ota B)$ onto
$\mcal{Z}(A) \ota \mcal{Z}( B)$ - see \cite{arc}. Making explicit use
of injectivity of $\oh$, the above natural map also extends to an
isometric algebra isomorphism from $\mcal{Z}(A \oh B)$ onto
$\mcal{Z}(A) \oh \mcal{Z}( B)$ - see \cite{ass}; and for $\oop$, it
extends to an algebraic $*$-isomorphism (not necessarily isometric)
from $\mcal{Z}(A \oop B)$ onto $\mcal{Z}(A) \oop \mcal{Z}( B)$ - see
\cite{JK}. With the kind of partial injectivity for $\obp$ established
in Section 2, their analogue for $\obp$ is quite satisfying.

\begin{theorem}
For $C^*$-algebras $A$ and $B$, $\mcal{Z}(A \obp B) = \mcal{Z}(A) \obp
\mcal{Z}( B).$
  \end{theorem}
\begin{proof}

Since $\Z(A \ot B) \subseteq \Z(A \obp B)$, consider the identity
function from $ \Z(A) \ot \Z(B)$ into $ \Z(A \obp B)$. By Theorem
\ref{obp-injective}, $\Z(A) \obp \Z(B)$ can be considered as a
$*$-subalgebra of $A\obp B$, so that for any $u\in \Z(A) \ot \Z(B)$, $
\|u\|_{\Z(A) \obp \Z(B)} = \|u\|_{A\obp B} $. Thus, the identity
function extends uniquely to an isometric $*$-homomorphism, say,
$\theta$ from $\Z(A) \obp \Z(B)$ into $\Z(A\obp B)$. It only remains
to show that $\theta $ is surjective.

Let $z \in \Z(A \obp B)$.  Consider $i: A\obp B \ra A\oh B$, the
canonical injective homomorphism as in \Cref{obp-min}. It is easily
seen that $i(z)x = x i(z)$ for all $x\in \Z(A) \ot \Z(B)$; so that $
i(z) \in \Z(A) \oh \Z(B)$, and by \cite{ass}, we have $ \Z(A\oh B) =
\Z(A) \oh \Z(B)$.  Now, let $i^\prime: \Z(A) \obp \Z(B) \ra \Z(A) \oh
\Z(B)$ be the canonical injective homomorphism (like the map $i$). Then, the
following diagram
$$
\xymatrix{
\Z(A)\obp \Z(B) \ar[rr]^\theta \ar[rd]_{i^\prime} && \Z(A\obp B) \ar[ld]^{i}\\
 & \Z(A)\oh \Z(B)} 
$$ commutates because $i \circ \theta = i^\prime $ on $\Z(A)
\ot \Z(B)$ and all the three maps are continuous.

Note that, the
map $i^\prime$ is surjective as well. To see this, consider an element
$z^\prime \in \Z(A) \oh \Z(B)$ and fix a sequence $\{z_n\} \seq
\Z(A)\ot \Z(B)$ such that $\|z_n -z^\prime\|_h \ra 0 $. By
Grothendieck inequality for commutative $C^*$-algebras (see
\cite{pis}), we have
$$
\|x\|_\gamma \leq K_G \|x\|_h \ \text{for all}\ x\in \Z(A)\ot
\Z(B),
$$ $K_G$ being the Grothendieck constant. Thus, the sequence $\{z_n\}$
is Cauchy with respect to $\|\cdot\|_{\gamma}$ and converges to some
$z^{\prime\prime}$ in $\Z(A) \obp \Z(B)$. This shows that $\{z_n =
i'(z_n)\}$ converges to $z'$ as well as to $z''$ in $\Z(A) \oh
\Z(B)$. So, $i^\prime(z^{\prime\prime}) = z^\prime$ and
$i^\prime$ is surjective.

Thus, for above $i(z)$ in  $\Z(A) \oh \Z(B)$,
there exists some $w \in \Z(A) \obp \Z(B)$ such that $i(z) =
i^\prime(w) = i( \theta(w))$. Since $i$ is injective, $z= \theta(w)$,
so that $\theta $ is surjective and we are done. 
\end{proof}

It would be interesting to provide an answer to the following:
\vspace*{1mm}

\noindent{\sc Question.} 
  Is there any relationship, as above, between the center of $A \obp B$ and
  $\Z(A) \obp \Z(B)$ for Banach algebras $A$ and $B$?


\end{document}